\documentclass[12pt]{article}
\usepackage{amssymb,amsfonts,amsmath,amsthm,dsfont,hyperref}
\usepackage[letterpaper, portrait, margin=27mm]{geometry}
\usepackage[lite,alphabetic]{amsrefs}

\newcommand{\1}{\mathds{1}}
\newcommand{\0}{\mathds{O}}

\newcommand{\R}{\mathbb{R}}

\newcommand{\N}{\mathbb{N}}

\newcommand{\8}{\infty}

\newcommand{\Co}{\mathcal{C}}
\newcommand{\Fo}{\mathcal{F}}
\newcommand{\Go}{\mathcal{G}}

\newcommand{\Po}{\mathcal{P}}
\newcommand{\Io}{\mathcal{I}}

\renewcommand{\mid}{\::\:}

\newcounter{dummy} \numberwithin{dummy}{section}
\newtheorem{theorem}[dummy]{Theorem}
\newtheorem{lemma}[dummy]{Lemma}

\newtheorem{proposition}[dummy]{Proposition}
\newtheorem{corollary}[dummy]{Corollary}

\theoremstyle{remark}
\newtheorem{remark}[dummy]{Remark}
\newtheorem{example}[dummy]{Example}

\DeclareMathOperator{\supp}{supp}

\begin{document}

\title{Atomicity of Boolean algebras and vector lattices in terms of order convergence}
\author{Antonio Avil\'{e}s, Eugene Bilokopytov and Vladimir G. Troitsky}
\maketitle

\begin{abstract}
We prove that order convergence on a Boolean algebra turns it into a compact convergence space if and only if this Boolean algebra is complete and atomic. We also show that on an Archimedean vector lattice, order intervals are compact with respect to order convergence if and only the vector lattice is complete and atomic. Additionally we provide a direct proof of the fact that uo convergence on an Archimedean vector lattice is induced by a topology if and only if the vector lattice is atomic.


\end{abstract}

\section{Introduction}

It is well-known that order convergence on a lattice is rarely \emph{topological} i.e., induced by a topology (see, e.g., \cite{erne}), which greatly reduces the toolkit of methods available for studying it. In the case when the lattice is a Boolean algebra or a vector lattice, the order convergence satisfies the definition of a \emph{convergence space} (see e.g. \cite{bb}), which is a broader notion than that of a topological space. The theory of convergence spaces is well developed, and its applications to the vector lattice theory has gained some attention recently (see e.g., \cite{aeg}, \cite{erz0}, \cite{otv} and \cite{ectv}).

This note is dedicated to topologicity and compactness (in the sense of convergence spaces) of order convergence in Boolean algebras and vector lattices. Let us list some of the existing results on this topic. The topologicity of order convergence on a Boolean algebra has been shown to be equivalent to atomicity in \cite{erne}. Order convergence on an Archimedean vector lattice can only be topological if the vector lattice is isomorphic to $\R^{n}$ for some $n\in\N$ (see \cite{dem}). A Boolean algebra admits a compact Hausdorff topology compatible with Boolean operations if and only if it is complete atomic if and only if it is isomorphic to a power set of some set. See \cite[Chapter~VII, Proposition~1.16]{johnstone} for this and other equivalent conditions; see also  \cite{bh}, \cite{weber} and \cite{zsw}. It follows from \cite[Theorem~1]{lipecky} that a vector lattice admits a locally solid Hausdorff topology in which order intervals are compact if and only if it is Dedekind complete and atomic, which in turn is equivalent to being an ideal in the Cartesian product of copies of~$\R$.

We complement these results by showing that a Boolean algebra is compact with respect to order convergence if and only if it is complete and atomic, and it is pre-compact if and only if it is atomic. For vector lattices, we show that order intervals are compact with respect to order convergence if and only if the vector lattice is Dedekind complete and atomic, answering a question from \cite{obr}. Intervals are precompact if and only if the vector lattice is atomic if and only if order convergence on intervals is topological.

\section{Preliminaries}

Throughout this section, $P$ is a lattice. By an \emph{order interval} in $P$ we mean a set of the form $\left[p,q\right]:=\left\{r\in P\mid p\le r\le q\right\}$, where $p,q\in P$ satisfy $p\le q$. We will also use open and semi-open versions of this notation, but do not refer to the corresponding sets as order intervals.

Recall that $Q\subseteq P$ is an \emph{ideal} if $p,q\in Q$ and $r\le p$ imply $r,p\vee q\in Q$. Dually, $Q$ is called a \emph{filter} if $p,q\in Q$ and $r\ge p$ imply $r,p\wedge q\in Q$. Maximal filters (i.e., filters $\varnothing\ne Q\varsubsetneq P$ which are not contained in any filter other than $P$ and $Q$) are called \emph{ultrafilters}. If $P$ is a Boolean algebra, the complement of $p\in P$ is denoted $\overline{p}$. We will need the following well-known fact in the sequel.

\begin{lemma}\label{ultr}
For any proper ideal $R$ in a Boolean algebra $P$ there is an ultrafilter $Q\subseteq P\backslash R$.
\end{lemma}
\begin{proof}
First, note that an ideal $R\subseteq P$ is proper iff it does not contain the greatest element $1$ of $P$. Therefore, the collection of all proper ideals in $P$ satisfies the conditions of Zorn's lemma. Let $R'$ be a maximal proper ideal of $P$ which contains $R$. It is easy to see that $Q:=\left\{p,~ \overline{p}\in R'\right\}$ is a maximal filter in $P$. Moreover, $Q\cap R'=\varnothing$, since if $p\in Q\cap R'$, then $p,\overline{p}\in R'$, hence $1=p\vee \overline{p}\in R'$, contradiction. Thus, $Q\subseteq P\backslash R'\subseteq P\backslash R$.
\end{proof}

If $\Omega$ is a set, by a filter on $\Omega$ we mean a filter in the power set $\Po\left(\Omega\right)$. We remind that if $\Omega$ is endowed with a topology, a filter on $\Omega$ is said to \emph{converge} to $\omega\in\Omega$ if it contains all neighborhoods of~$\omega$. In particular, the intersection of any collection of filters that converge to $\omega$ also converges to~$\omega$.

A net $\left(p_{\alpha}\right)_{\alpha\in A}$ in $P$ \emph{order converges} to $p\in P$ (denoted $p_{\alpha}\to p$) if there exist an increasing net $\left(q_{\beta}\right)_{\beta\in B}$ and a decreasing net $\left(r_{\beta}\right)_{\beta\in B}$ such that $\bigvee\limits_{\beta\in B}q_{\beta}=p=\bigwedge\limits_{\beta\in B}r_{\beta}$ and for every $\beta\in B$ there is $\alpha_{0}\in A$ such that $q_{\beta}\le p_{\alpha}\le r_{\beta}$, for every $\alpha\ge\alpha_{0}$. One can show that order limits are unique. Clearly, if $\left(p_{\alpha}\right)_{\alpha\in A}$ is increasing then $p_{\alpha}\to p$ iff $p=\bigvee p_{\alpha}$ (and the same for the infimum).

A filter $\Fo$ on $P$ \emph{order converges} to $p$ if it contains a collection of order intervals in $P$, whose intersection is equal to $\left\{p\right\}$. The reason why the term ``order convergence'' is used in two contexts is that these are two manifestations of order convergence, one in the language of nets, and the other -- in the language of filters. Namely, a filter is order convergent to $p$ iff it is generated by the tails of a net which order converges to $p$ (a \emph{tail} of a net $\left(p_{\alpha}\right)_{\alpha\in A}$ is a set of a form $\left\{p_{\alpha}\right\}_{\alpha\ge\alpha_{0}}$, for some $\alpha_{0}\in A$; we will also use the term ``tail'' for the net  $\left(p_{\alpha}\right)_{\alpha\ge\alpha_{0}}$, the ambiguity will be cleared up by the context). For further details on the correspondence between these languages, see in \cite{otv}. We say that order convergence is \emph{topological} if there is a topology on $P$ such that order convergence is induced by this topology. For a characterization of lattices in which order convergence is topological see~\cite{erne}.

We say that $P$ is \emph{order compact} if it is compact with respect to the order convergence structure (see \cite[Section~1.4]{bb}), i.e., if every ultrafilter on $P$ is order convergent to an element of~$P$. In the language of nets this translates to the following property: every net in $P$ has an order convergent subnet. For a proof of the equivalence of these properties see \cite[p.~246 and p.~258]{otv}. It is easy to see that a subnet of an order convergent net order converges to the same limit. We will say that $P$ is \emph{order locally compact} if it is locally compact with respect to the order convergence structure, i.e., if every order convergent filter in $P$ contains an order compact set. In the language of nets this translates to the property that every order convergent net has a tail contained in an order compact set. Clearly, order compactness implies order local compactness. We start with some general observations.

\begin{lemma}Let $I\subseteq P$ be an order interval. Then:
\item[(i)] If $\left(p_{\alpha}\right)_{\alpha\in A}\subseteq I$ and $p_{\alpha}\to p$ then $p\in I$.
\item[(ii)] Order convergence in $I$ agrees with order convergence in $P$. That is, if $\left(p_{\alpha}\right)_{\alpha\in A}\subseteq I$ and $p\in I$, then $p_{\alpha}\to p$ in $P$ iff $p_{\alpha}\to p$ in~$I$.
\item[(iii)] If $P$ is order compact, then so is~$I$.
\end{lemma}

Note that the lemma is still valid if $I$ is an ``unbounded order interval'', i.e., a set of the form $\left\{p\in P\mid p\le q\right\}$ or $\left\{p\in P\mid p\ge q\right\}$, for $q\in P$.\medskip

Since a subnet of an increasing net is increasing and has the same supremum as the original net, we have the following result.

\begin{lemma}
If a monotone net in $P$ has an order convergent subnet then it converges to the same limit.
\end{lemma}

\begin{proposition}\label{completel}
If $P$ is order compact, then it is a complete lattice.
\end{proposition}
\begin{proof}
For $Q\subseteq P$ consider the set $Q^{\vee}:=\left\{q_{1}\vee...\vee q_{n},~ q_{1},..., q_{n}\in Q,~n\in\N\right\}$, which can be viewed as an increasing net indexed by itself. Since $P$ is order compact, this net contains an order convergent subnet, hence order converges itself, therefore has a supremum, which also serves as the supremum of~$Q$. Consequently, every subset of $P$ has a supremum, and analogously every subset of $P$ has an infimum, thus $P$ is a complete lattice.
\end{proof}

It is well-known that if $P$ is a complete lattice then order convergence of nets can be described as follows: a net $\left(p_{\alpha}\right)_{\alpha\in A}$ order converges to $p\in P$ iff $p=\bigvee\limits_{\alpha}\bigwedge\limits_{\beta\ge \alpha}p_{\beta}=\bigwedge\limits_{\alpha}\bigvee\limits_{\beta\ge \alpha}p_{\beta}$. Moreover, in the general case, order convergence on a lattice $P$ is the restriction of the order convergence on its Dedekind-MacNeille completion (see e.g. \cite[Theorem~3]{acw}, \cite[Theorem~2.5]{pap1}). If $P$ is an Archimedean vector lattice then order convergence on $P$ is the restriction of the order convergence on its Dedekind completion (see \cite[Proposition~1.5]{ab}). We remark here that Dedekind-MacNeille completion of a vector lattice is the Dedekind completion supplemented by $\pm\8$, so formally the latter result is not a special case of the one mentioned in the previous sentence. The following result is known, but we include it for completeness.

\begin{proposition}\label{pointwise}
Let $\Omega$ be a set, $\R^{\Omega}$ the space of all real-valued functions on $\Omega$, and $P$ an order interval in $\R^{\Omega}$. Then order convergence on $P$ coincides with pointwise convergence (and so is compact and topological).
\end{proposition}
\begin{proof}
Since the supremum (infimum) of any collection of functions is their pointwise supremum (infimum), it follows that $P$ is a complete lattice, and so for a net $\left(f_{\alpha}\right)_{\alpha\in A}\subseteq P$ and $f\in P$ we have $f_{\alpha}\to f$ iff $f=\bigvee\limits_{\alpha}\bigwedge\limits_{\beta\ge \alpha}f_{\beta}=\bigwedge\limits_{\alpha}\bigvee\limits_{\beta\ge \alpha}f_{\beta}$ which is equivalent to $f\left(\omega\right)=\bigvee\limits_{\alpha}\bigwedge\limits_{\beta\ge \alpha}f_{\beta}\left(\omega\right)=\bigwedge\limits_{\alpha}\bigvee\limits_{\beta\ge \alpha}f_{\beta}\left(\omega\right)$, for every $\omega\in\Omega$, i.e., $f$ is the pointwise limit of $\left(f_{\alpha}\right)_{\alpha\in A}$.

Note that pointwise topology on $\R^{\Omega}$ is the product topology. Hence, if $P=\left[f,g\right]$ for some $f,g\in\R^{\Omega}$ then $P=\prod\limits_{\omega\in\Omega}\left[f\left(\omega\right),g\left(\omega\right)\right]$ is a compact topological space.
\end{proof}

\begin{remark}\label{pintwise}
Similarly to the preceding proposition, order convergence on $\Po\left(\Omega\right)\simeq \left\{0,1\right\}^{\Omega}$ is pointwise or product topology, in particular it is compact.
\qed\end{remark}

Now assume that $P$ is either a Boolean algebra or an interval in a vector lattice. In the case of Boolean algebra we will use the ``$-$'' symbol for the symmetric difference. For $A,B\subset P$ we denote $A-B:=\left\{a-b,~a\in A,~b\in B\right\}$. A net $\left(p_{\alpha}\right)_{\alpha\in A}\subseteq P$ is called \emph{order Cauchy} if the double indexed net $\left(p_{\alpha}-p_{\beta}\right)_{\alpha,\beta\in A}$ is order null, i.e., $p_{\alpha}-p_{\beta}\to 0$. It follows from \cite[Corollary~6.5]{et} and \cite[Proposition~6.1]{pap1} that this condition implies existence of a proper order interval in $P$ which contains a tail of $\left(p_{\alpha}\right)_{\alpha\in A}$.

A filter $\Fo$ on $P$ is \emph{order Cauchy} if the filter $\Fo-\Fo$ generated by the family\linebreak $\left\{A-B,~ A,B\in\Fo\right\}$ is order null. Note that if $\Fo$ is generated by the tails of a net $\left(p_{\alpha}\right)_{\alpha\in A}$, then the filter $\Fo-\Fo$ is generated by the tails of $\left(p_{\alpha}-p_{\beta}\right)_{\alpha,\beta\in A}$. Hence, a filter is order Cauchy iff it is generated by the tails of an order Cauchy net. It follows from the preceding paragraph that every order Cauchy filter contains a proper interval in~$P$.

Clearly, every order convergent net or filter is order Cauchy. It is known that a Boolean algebra or an order interval in an Archimedean vector lattice is a complete lattice iff it is order complete, i.e., every order Cauchy net (filter) is convergent (see \cite[Proposition~9.10]{otv} or \cite[Theorem~6.8]{pap1}; the proof for Boolean algebras is similar). Recall that a vector lattice is called \emph{Dedekind complete} if order intervals in it are complete lattices.

We will call $P$ \emph{order pre-compact} if every ultrafilter on $P$ is order Cauchy. Using the same arguments as for the characterization of order compactness one can show that $P$ is order pre-compact iff every net in $P$ has an order Cauchy subnet. It follows from Proposition \ref{completel} that $P$ is order compact iff it is order complete and order pre-compact.

Finally, $P$ is called \emph{order locally pre-compact} if every convergent net has an order pre-compact tail, or equivalently, every convergent filter in $P$ contains an order pre-compact set. Order local pre-compactness follows from order pre-compactness, as well as from order local compactness. All the three properties are inherited by order intervals.

Note that Boolean operations are continuous with respect to order convergence, hence a Boolean algebra is a convergence group. Therefore, an order locally compact Boolean algebra is complete with respect to order convergence (see \cite[Proposition~3.1.14]{bb}), and so is a complete lattice.

\section{The case of a Boolean algebra}

Throughout this section, $A$ is a Boolean algebra with the least and greatest elements $0$ and $1$, respectively. We call $a,b\in A$ \emph{disjoint}, and write $a\bot b$ if $a\wedge b=0$. Note that any order interval in $A$ is a Boolean algebra with respect to the induced order structure.

Recall that $a\in A$ is an \emph{atom} if $a>0$ and $\left[0,a\right]=\left\{0,a\right\}$. We will say that $A$ is \emph{atomless} if it contains no atoms, and \emph{atomic} if $\left[0,b\right]$ contains an atom for every $b>0$.

$A$ is atomic iff its completion is isomorphic to the powerset $\Po\left(\Omega\right)$ for some set $\Omega$ (see \cite[35.2 Example~A]{sikorski}). In particular, $A$ is complete and atomic iff it is isomorphic to $\Po\left(\Omega\right)$ for some set~$\Omega$.

If $A$ is atomless, and $a>0$ then there exists $b\in\left(0,a\right)$; then $b$ and $a-b$ are disjoint and non-zero, and $a=\left(a-b\right)\vee b$. Furthermore, if $c<d$, then $\left[c,d\right]$ is an atomless Boolean algebra. Note that if $A$ is not atomic then there is $a>0$ such that $\left[0,a\right]$ contains no atoms, and so is an atomless Boolean algebra.

\begin{lemma}\label{atomless}
If $A$ is atomless then it cannot be covered by a finite collection of proper (i.e., not equal to $A$) order intervals.
\end{lemma}
\begin{proof}
Assume that there are $a_{1},...,a_{n},b_{1},...,b_{n}$ in $A$ such that $a_{k}\le b_{k}$ for every $k=1,...,n$ and $A=\left[a_{1},b_{1}\right]\cup...\cup \left[a_{n},b_{n}\right]$. Note that it then follows that one of $a_{k}$'s is equal to~$0$, and one of $b_{k}$'s is equal to~$1$. There is a finite collection $U=\left\{u_{1},...,u_{m}\right\}\subseteq A$ of disjoint non-zero elements such that every $a_{k}$ and $b_{k}$ can be represented as the supremum of a subset of $U$ (and so in particular $u_{1}\vee...\vee u_{m}=1$). Since $A$ is atomless, there are disjoint non-zero $v_{1},...,v_{m},w_{1},...,w_{m}\subseteq A$ such that $u_{k}=v_{k}\vee w_{k}$, for every $k=1,...,m$. Then, there is $l\in 1,...,n$ such that $v:=v_{1}\vee...\vee v_{m}\in \left[a_{l},b_{l}\right]$. Since $a_{l}$ is the supremum of a subcollection of~$U$, but $v\bot w_{k}$, hence $v\not\ge u_{k}$, for every $k=1,...,m$, it follows that $a_{l}=0$. Similarly, since $b_{l}$ is the supremum of a subcollection of~$U$, and $b_{l}\ge v_{k}$, for every $k=1,...,m$, it follows that $b_{l}= u_{1}\vee...\vee u_{m}=1$. Thus, $\left[a_{l},b_{l}\right]=A$.
\end{proof}

We are now ready to characterize order compact and order pre-compact Boolean algebras.

\begin{theorem}\label{bain}For a Boolean algebra $A$ the following conditions are equivalent:
\item[(i)] $A$ is atomic;
\item[(ii)] $A$ is order pre-compact;
\item[(iii)] $A$ is order locally pre-compact.
\end{theorem}
\begin{proof}
(i)$\Rightarrow$(ii): If $A$ is atomic, its  completion is isomorphic to $\Po\left(\Omega\right)$, for some set~$\Omega$. According to Remark~\ref{pintwise}, on $\Po\left(\Omega\right)\simeq\left\{0,1\right\}^{\Omega}$ order convergence coincides with the convergence in the product topology in $\left\{0,1\right\}^{\Omega}$, which is compact. Since order convergence on a lattice is a restriction of order convergence on its completion, it follows that order convergence on $A$ is the restriction of a compact topology, and so is pre-compact.\medskip

(ii)$\Rightarrow$(iii) is obvious.\medskip

(ii)$\Rightarrow$(i): Assume that $A$ is order pre-compact. If it is not atomic then there is $a>0$ such that $\left[0,a\right]$ is an atomless Boolean algebra, which is also order pre-compact. Hence, in order to reach a contradiction we may assume without loss of generality that $A$ is atomless. Let $\Io$ be the collection of all subsets of $A$ that can be covered by a finite union of proper order intervals in~$A$. Clearly, $\Io$ is an ideal in $\Po\left(A\right)$. Moreover, according to Lemma \ref{atomless}, $A\notin\Io$. Hence, $\Io$ is a proper ideal, and so according to Lemma~\ref{ultr} there is an ultrafilter $\mathcal{U}\subseteq \Po\left(A\right)\backslash \Io$. Observe that $\mathcal{U}$ contains no proper order intervals of $A$ and so is not order Cauchy. This contradicts order pre-compactness.\medskip

(iii)$\Rightarrow$(i): Assume that $A$ is order locally pre-compact, but not atomic. Let $a>0$ be such that $\left[0,a\right]$ is an atomless Boolean algebra. Applying Lemma~\ref{ultr} to the ideal $\left\{0\right\}$ of $\left[0,a\right]$ yields existence of an ultrafilter $U\subseteq \left(0,a\right]$ of the Boolean algebra $\left[0,a\right]$.

We claim that $\bigwedge U=0$. Indeed, if $U\ge b>0$ then there is $c\in\left(0,b\right)$, hence $\left[c,a\right]$ is a filter in $\left[0,a\right]$ with $U\varsubsetneq \left[c,a\right] \varsubsetneq \left[0,a\right]$, which contradicts maximality of~$U$. Hence, $U$ can be viewed as a net which decreases (hence order converges) to~$0$. Since $A$ is order locally pre-compact, there is $u\in U$ such that $\left[0,u\right]\cap U$ is order pre-compact. As $\bigwedge U=0< u$, there is $v\in U$ such that $v\not\ge u$, hence $w:=u\wedge w\in U\cap \left(0,u\right)$. Then $\left[w,a\right]\subseteq U$, and in particular $\left[w,u\right]\subseteq \left[0,u\right]\cap U$. It follows that $\left[w,u\right]$ is simultaneously order pre-compact, and an atomless Boolean algebra with respect to the induced order. Contradiction.
\end{proof}

\begin{corollary}\label{bainc}For a Boolean algebra $A$ the following conditions are equivalent:
\item[(i)] $A$ is complete and atomic;
\item[(ii)] $A$ is order compact;
\item[(iii)] $A$ is order locally compact.
\end{corollary}
\begin{proof}
(i)$\Rightarrow$(ii) follows from Remark~\ref{pintwise} and the fact that every complete and atomic Boolean algebra is isomorphic to $\Po\left(\Omega\right)$ for some set~$\Omega$. (ii)$\Rightarrow$(iii) is obvious.

(iii)$\Rightarrow$(i): If $A$ is order locally compact, then it is order locally pre-compact, and so atomic, according to Theorem~\ref{bain}. On top of that, as was mentioned above, order local compactness implies completeness.
\end{proof}

As was mentioned in the introduction, there are numerous other characterizations of complete atomicity (see \cite[Chapter~VII, Proposition~1.16]{johnstone}). It was proven in \cite{erne} that order convergence in a Boolean algebra is topological iff the Boolean algebra is atomic. We provide a direct alternative proof of this result.

\begin{proposition}[\cite{erne}]\label{boot}
Order convergence in a Boolean algebra is topological iff the Boolean algebra is atomic.
\end{proposition}
\begin{proof}
Necessity is proven in the same way as (i)$\Rightarrow$(ii) in Theorem \ref{bain}. In order to prove sufficiency we assume that $A$ is a Boolean algebra which is not atomic, but such that order convergence on $A$ is topological. Let $a>0$ such that $\left[0,a\right]$ is atomless. Then order convergence on $\left[0,a\right]$ is topological.

Represent $\left[0,a\right]$ as the Boolean algebra of clopen subsets of a compact totally disconnected Hausdorff space~$K$. Since isolated points of $K$ correspond to atoms of $\left[0,a\right]$ we conclude that $K$ has no isolated points. For every $x\in K$ consider the collection $B_{x}\subset A$ of clopen subsets of $K$ which do not contain~$x$. Since $x$ is not isolated, it follows that $\bigvee B_{x}=K$, where the supremum is taken in $Clop\left(K\right)$. Therefore, $\bigcap\limits_{C\in B_{x}}\left[C,K\right]=\left\{K\right\}$, and so the filter $\Go_{x}$ on $\left[0,a\right]$ generated by $\left\{\left[C,\1\right],~C\in B_{x}\right\}$ order converges to~$K$. To prove that order convergence is not topological it is enough to show that $\Go:=\bigcap\limits_{x\in K}\Go_{x}$ does not order converge to~$K$. That is true because $\Go$ does not contain any proper order intervals in $\left[0,a\right]$. Indeed, if $\left[C,D\right]$ is a member of $\Go$ then for every $x\in K$ we have that $\left[C,D\right]\in\Go_{x}$. Hence, there is $C_{x}\in B_{x}$ such that $\left[C_{x},K\right]\subseteq \left[C,D\right]$, from where $D=K$, and $C\subseteq C_{x}\subset K\backslash\left\{x\right\}$. Since $x$ was arbitrary we conclude that $C=\varnothing$.
\end{proof}

Combining Theorem~\ref{bain} with Proposition~\ref{boot}, we arrive at the conclusion that a Boolean algebra is order pre-compact iff order convergence on it is topological.

\section{The case of a vector lattice}

We now turn to vector lattices, where we follow a similar scheme as in the preceding section. Throughout this section, $F$ is an Archimedean vector lattice. We start with an analogue of Lemma~\ref{atomless}. If $K$ is a topological space, let $\Co\left(K\right)$ be the vector lattice of all continuous functions on $K$ with pointwise operations. We also denote the constant functions on $K$ with values $0$ and $1$ by $\0$ and~$\1$, respectively. A \emph{linear sublattice} of a vector lattice is a linear subspace which is also closed with respect to lattice operations. Note that in vector lattice literature, the term ``sublattice'' usually means linear sublattice. If $F$ is a vector lattice, then we call $e,f\in F$ \emph{disjoint} and write $e\bot f$ if $\left|e\right|\wedge \left|f\right|=0$.

\begin{lemma}\label{atomles}
Let $K$ be a compact Hausdorff space with no isolated points. If $F\subseteq\Co\left(K\right)$ is a norm dense linear sublattice then $F\cap \left[\0,\1\right]$ cannot be covered by a finite collection of proper order intervals.
\end{lemma}
\begin{proof}
  Denote $P:=F\cap \left[\0,\1\right]$. Assume that $e_{1},...,e_{n},f_{1},...,f_{n}$ are such that $P=\left[e_{1},f_{1}\right]\cup...\cup\left[e_{n},f_{n}\right]$. Let also $g_{k}:=\1-f_{k}$, for every $k=1,...,n$, and $A:=\left\{e_{1},...,e_{n}\right\}$, and $B:=\left\{g_{1},...,g_{n}\right\}$. Let $A^{\wedge}$ be the collection of all infima of elements of $A$. Let $e'_{1},...,e'_{m}$ be the minimal elements of $A^{\wedge}\backslash\left\{\0\right\}$. Clearly, these functions are disjoint, and for any $e_{i}$ and $e'_{j}$ we have that either $e_{i}\bot e'_{j}$ or $e'_{j}\le e_{i}$. Moreover, since $A^{\wedge}\backslash\left\{\0\right\}$ is finite, for every $e_{i}\ne\0$ there is at least one $j$ such that $e'_{j}\le e_{i}$. Define $g'_{1},...,g'_{l}$ to be the minimal elements of $B^{\wedge}\backslash\left\{\0\right\}$; they have similar properties.

Pick $x_{i}\in\supp e'_{i}$, for every $i=1,...,m$. Since there are no isolated points, for every $j=1,...,l$ there is $y_{j}\in\supp g'_{j}\backslash \left\{x_{1},...,x_{m}\right\}$. Surround $\left\{x_{1},...,x_{m}\right\}$ with an open $U$ such that $y_{j}\notin\overline{U}$, for every~$j$. Then surround $\left\{y_{1},...,y_{l}\right\}$ with an open $V$ such that $\overline{U}\cap\overline{V}=\varnothing$.

Let $u:=\frac{1}{2}\left(e'_{1}+...+e'_{m}\right)$ and $v:=\1 - \frac{1}{2}\left(g'_{1}+...+g'_{l}\right)$. From sublattice Uryshohn lemma (see \cite[Proposition~2.1]{erz}) applied to $\left\{x_{1},...,x_{m}\right\}\subseteq U$ and $u$ there is $e\in P$ which vanishes outside of $U$ and such that  $e\left(x_{i}\right)=u\left(x_{i}\right)=\frac{1}{2}e'_{i}\left(x_{i}\right)$, for every~$i$. Similarly, applying the lemma to $\left\{y_{1},...,y_{l}\right\}\subseteq V$ and $v$ yields
$g\in P$ which vanishes outside of $V$ and such that $g\left(y_{j}\right)=v\left(y_{j}\right)=1-\frac{1}{2}g'_{j}\left(y_{j}\right)$, for every~$j$. Note that $e\bot g$.

Let $h:=e+g$. WLOG we may assume that $h\in \left[e_{1},f_{1}\right]$. Suppose $e_{1}\ne\0$; after relabeling $e'_{1},...,e'_{m}$, WLOG $e'_{1}\le e_{1}$, and so $e_{1}\left(x_{1}\right)\le h\left(x_{1}\right)=e\left(x_{1}\right)=\frac{1}{2}e'_{1}\left(x_{1}\right)<e'_{1}\left(x_{1}\right)\le e_{1}\left(x_{1}\right)$, contradiction. Analogously, if $f_{1}\ne\1$ then WLOG $g'_{1}\le g_{1}=\1-f_{1}$, and so $f_{1}\left(y_{1}\right)\ge h\left(y_{1}\right)= g\left(y_{1}\right)=1-\frac{1}{2}g'_{1}\left(y_{1}\right)>1- g'_{1}\left(y_{1}\right)\ge f_{1}\left(y_{1}\right)$, contradiction. Thus, $\left[e_{1},f_{1}\right]=P$.
\end{proof}

A (linear) sublattice $E\subseteq F$ is called \emph{order dense} if for every $f>0$ there is $e\in E$ with $0<e\le f$. A linear subspace $E\subseteq F$ is called an \emph{ideal} if $f\in F$ and $e\in E$ with $\left|f\right|\le \left|e\right|$ imply $f\in E$. Every ideal is a sublattice. If an order dense sublattice of a vector lattice is Dedekind complete with respect to the induced order structure, then it is an ideal (see \cite[Theorem~1.40]{ab0}). Both ideals and order dense sublattices have a property that order convergence on their order intervals is the restriction of order convergence of $F$ (for ideals this is easy to see, for order dense sublattices combine \cite[Theorem~1.23]{ab} with \cite[Corollary~2.12]{gtx}). The intersection of any collection of ideals is an ideal; the smallest ideal which contains $f\in F_{+}$ is denoted by~$I_{f}$. There is a vector lattice isomorphism $T$ from $I_{f}$ onto a norm dense sublattice of $\Co\left(K\right)$, where $K$ is a compact Hausdorff space, such that $Tf=\1$ (this follows from Krein-Kakutani theorem, see e.g., \cite[Theorem~4.21]{ab}). For any $G\subseteq F$ its \emph{disjoint complement} $G^{d}:=\left\{f\in F,~\forall g\in G:~ f\bot g\right\}$ is an ideal.\medskip

Recall that $f>0$ is called an \emph{atom} if $\left[0,f\right]$ consists of scalar multiples of~$f$. If $E\subset F$ is an ideal, then $e\in E$ is an atom in $E$ iff it is an atom in~$F$. We call $F$ \emph{atomic} if $\left[0,g\right]$ contains an atom, for every $g>0$. In this case, $F$ is isomorphic to an order dense sublattice of $\R^{\Omega}$ for some set $\Omega$ (see \cite[Theorem~1.78]{ab0}); if $F$ is Dedekind complete then this sublattice is, in fact, an ideal of~$\R^{\Omega}$. Note that order density in $\R^{\Omega}$ amounts to containing characteristic functions of all singletons in~$\Omega$. It is easy to see (cf. \cite[p.~143]{schaefer}) that the set $F_{a}$ of all $f\in F$ such that $\left|f\right|$ is a supremum of a collection of atoms is an ideal in~$F$. Moreover, $F_{a}$ is an atomic vector lattice (we call it the \emph{atomic part} of $F$). Its disjoint complement $F_{a}^{d}$ is an ideal in $F$ which contains no atoms (we call it the \emph{atomless part} of $F$), and $F_{a}=F_{a}^{dd}$. If $f\in F_{a}$ then $I_{f}\subseteq F_{a}$, and so $I_{f}$ is an atomic vector lattice. Similarly, if $f\in F_{a}^{d}$ then $I_{f}$ is an atomless vector lattice. If $0\le f\notin F_{a}$ then $f\not\perp F_{a}^{d}$, and so there is $g\in F_{a}^{d}\cap\left(0,f\right]$.\medskip

Let $S$ be a Tychonoff space. The space $\Co_{b}\left(S\right)$ of all bounded continuous functions on $S$ is an ideal in $\Co\left(S\right)$.

\begin{lemma}\label{cba}
Let $S$ be a Tychonoff space, and let $F$ be a norm dense sublattice of $\Co_{b}\left(S\right)$. Then, the atoms in $F$ are positive scalar multiples of the characteristic functions of isolated points in~$S$. Moreover, $F$ is atomic iff $S$ has a dense set of isolated points.
\end{lemma}
\begin{proof}
Let $s\in S$ be an isolated point and let $f:=\1_{\left\{s\right\}}$. According to sublattice Uryshohn lemma (see \cite[Proposition~2.1]{erz}) applied to $f$, $F$, and the compact and open set $\left\{s\right\}$, it follows that $f\in F$. It is easy to see that $rf$ is an atom in $F$ for any $r>0$.

Conversely, assume that $f\in F_{+}$ is not a positive scalar multiple of a characteristic function of an isolated point in~$S$. Then,  $f\left(s\right),f\left(t\right)>0$, for some distinct $s,t\in S$. Let $U$ be a neighborhood of $s$ which does not contain~$t$. Applying the sublattice Uryshohn lemma to $f$, $F$, $\left\{s\right\}$, and $U$ yields $e\in F \cap \left[\0,f\right]$ which vanishes outside of $U$ such that $e\left(s\right)=f\left(s\right)>0$. Clearly, $e$ cannot be a scalar multiple of~$f$, and so $f$ is not an atom in~$F$.

It is easy to see that if $S$ has a dense set of isolated points, then there is an atom in $ F \cap \left[\0,f\right]$, for every $f>0$. Assume that isolated points do not form a dense set in~$S$. Then, there is an open $U\subset S$ which contains no isolated points of~$S$. Using Urysohn Lemma again, one can construct $\0<f\in F_{+}$ which vanishes outside~$U$. It follows that $F\cap \left[\0,f\right]$ contains no atoms.
\end{proof}

We can now prove a ``local version'' of Theorem~\ref{bain} for vector lattices.

\begin{theorem}\label{topuo}
For $f\in F_{+}$ the following conditions are equivalent:
\item[(i)] $f\in F_{a}$;
\item[(ii)] $\left[0,f\right]$ is order pre-compact;
\item[(iii)] Order convergence on $\left[0,f\right]$ is topological.
\end{theorem}
\begin{proof}
(i)$\Rightarrow$(ii),(iii): Since order convergence on the order intervals is the restriction of order convergence on~$F$, without loss of generality we may assume that $F=I_{f}$, which is atomic, according to the discussion before Lemma~\ref{cba}. Hence, $F$ can be identified with an order dense linear sublattice of~$\R^{\Omega}$, for some set~$\Omega$. Order convergence on $\left[0,f\right]$ in $F$ is therefore the restriction of order convergence on an order interval in~$\R^{\Omega}$. By Proposition~\ref{pointwise}, this is the pointwise convergence, hence topological and pre-compact.\medskip

(ii) or (iii)$\Rightarrow$(i): Assume that (i) fails, that is $f\notin F_{a}$. As was mentioned before the theorem, there is $g\in F_{a}^{d}\cap\left(0,f\right]$. Let $T$ be a vector lattice isomorphism from $I_{g}$ onto a norm dense sublattice of $\Co\left(K\right)$, where $K$ is a compact Hausdorff space, such that $Tg=\1$. Since $g$ is in the atomless part of $F$, there are no atoms in~$I_{g}$, and so, according to Lemma~\ref{cba}, $K$ has no isolated points. Therefore, $\left[0,g\right]$ is order isomorphic to $P:=\left[\0,\1\right]\cap TI_{g}$. As order topologicity and order pre-compactness of $\left[0,f\right]$ would trickle down to $\left[0,g\right]$, in order to show that conditions (ii) and (iii) fail, it is enough to prove the corresponding statements for~$P$.

For pre-compactness, we apply the same argument as in Theorem~\ref{bain}. Consider an ideal $\Io$ of sets in $P$ which can be covered by a finite number of proper order intervals. According to Lemma~\ref{atomles}, this ideal is proper. Hence, due to Lemma~\ref{ultr}, there is an ultrafilter $\mathcal{U}$ on $P$ which is disjoint with~$\Io$. It follows that $\mathcal{U}$ contains no proper intervals, thus cannot be  order Cauchy.

Let us now prove that order convergence is not topological on~$P$. This time we will mimic the argument from Proposition~\ref{boot}. For every $x\in K$ let $G_{x}$ be the set of all functions in $P$ that vanish at~$x$. As $x$ is not isolated, it follows that $\bigvee G_{x}=\1$, and so $\bigcap\limits_{u\in G_{x}}\left[u,\1\right]=\left\{\1\right\}$. Therefore, the filter $\Fo_{x}$ on $P$ generated by $\left\{\left[u,\1\right],~u\in G_{x}\right\}$ order converges to~$\1$. To prove that order convergence is not topological it is enough to show that $\Fo:=\bigcap\limits_{x\in K}\Fo_{x}$ does not order converge to~$\1$. That is true because $\Fo$ does not contain any proper order intervals in~$P$. Indeed, if $\left[v,w\right]$ is a member of $\Fo$ then for every $x\in K$ we have that $\left[v,w\right]\in\Fo_{x}$. Hence, there is $u_{x}\in G_{x}$ such that $\left[u_{x},\1\right]\subseteq \left[v,w\right]$, from where $w=\1$, and $v\le u_{x}$. In particular, $v$ vanishes at~$x$; since $x$ was arbitrary we conclude that $v=\0$.
\end{proof}

\begin{corollary}\label{orlco}
For $f\in F_{+}$, the order interval $\left[0,f\right]$ is order locally pre-compact iff $f\in F_{a}$.
\end{corollary}
\begin{proof}
In the light of the theorem, it is enough to show that $\left[0,f\right]$ order locally pre-compact iff it is order pre-compact. Clearly, order pre-compactness implies order local pre-compactness. For the converse, note that due to Archimedean property, the filter generated by $\left\{\left[0,\frac{1}{n}f\right]\mid n\in\N\right\}$ is order null, hence contains an order pre-compact set. This set has to contain $\left[0,\frac{1}{n}f\right]$ for some $n\in\N$, thus $\left[0,\frac{1}{n}f\right]$ is order pre-compact. As $\left[0,\frac{1}{n}f\right]$ and $\left[0,f\right]$ are isomorphic lattices, the latter is also order pre-compact.
\end{proof}

The ``global version'' of Theorem~\ref{bain} for vector lattices is an immediate consequence of Theorem~\ref{topuo}.

\begin{theorem}\label{topuo2}
The following conditions are equivalent:
\item[(i)] $F$ is atomic;
\item[(ii)] Every order interval in $F$ is order pre-compact;
\item[(iii)] Order convergence on order intervals of $F$ is topological.
\end{theorem}

\begin{corollary}
$F$ is atomic iff it is order locally pre-compact.
\end{corollary}
\begin{proof}
Assume that $F$ is order locally pre-compact. As order local pre-compactness is inherited by order intervals, it follows that  $\left[0,f\right]$ is order locally pre-compact for every $f\in F_{+}$. By Corollary \ref{orlco}, this means that $F$ is atomic.

Conversely, if $F$ is atomic then, according to~Theorem \ref{topuo2}, every order interval in $F$ is order pre-compact. As every order convergent net has a tail contained in an order interval, we conclude that $F$ is order locally pre-compact.
\end{proof}

We say that a net $\left(f_{\alpha}\right)_{\alpha\in A}\subseteq F$ \emph{unbounded order (uo)} converges to $f\in F$ if $\left|f_{\alpha}-f\right|\wedge h\to 0$, for every $h\in F_{+}$. For some background information on uo convergence, see \cite{gtx} and \cite{erz0}; for uo convergence on spaces of continuous functions, see \cite{et}. It is easy to deduce from Proposition~\ref{pointwise} that on $\R^{\Omega}$ uo convergence coincides with the pointwise convergence.

\begin{corollary}
$F$ is atomic iff uo convergence on $F$ is topological.
\end{corollary}
\begin{proof}
Necessity follows from the fact that on an order dense sublattice of $\R^{\Omega}$ uo convergence coincides with the pointwise convergence. Sufficiency follows from Theorem~\ref{topuo2} and the fact that uo convergence coincides with order convergence on the order intervals.
\end{proof}

\begin{remark}\label{ellis}
The preceding result already exists in the literature, but in an indirect way: it was proven in \cite{ellis} that the topologicity of the so-called alpha-convergence on an l-group is equivalent to the fact that the l-group is completely distributive; in \cite{weinberg} it was established that complete distributivity of an Archimedean l-group is equivalent to its atomicity, while in Archimedean l-groups alpha-convergence coincides with the uo convergence, according to \cite{pap2}. There is also a direct proof in \cite[Theorem~6.45]{taylor}, however, it is longer and more technical.
\qed\end{remark}

\begin{theorem}\label{topuo3}
For a vector lattice~$F$, the following conditions are equivalent:
\item[(i)] $F$ is Dedekind complete and atomic;
\item[(ii)] Every order interval in $F$ is order compact;
\item[(iii)] $F$ is order locally compact.
\end{theorem}
\begin{proof}
(i)$\Rightarrow$(ii): As was mentioned above, if $F$ is Dedekind complete and atomic, it is isomorphic to an ideal in~$\R^{\Omega}$, for some set~$\Omega$. Then, every order interval in $F$ is isomorphic to an order interval in~$\R^{\Omega}$, which is order compact, according to Proposition~\ref{pointwise}.

(ii)$\Rightarrow$(i): By Proposition~\ref{completel}, every order interval in $F$ is a complete lattice, hence $F$ is Dedekind complete. It follows from Theorem~\ref{topuo2} that $F$ is atomic.

(ii)$\Rightarrow$(iii) follows from the fact that every convergent net has a tail contained in an order interval.

(iii)$\Rightarrow$(ii): Fix $f\in F_{+}$. The filter generated by $\left\{\left[0,\frac{1}{n}f\right]\mid n\in\N\right\}$ is order null, hence contains an order compact set. This set has to contain $\left[0,\frac{1}{n}f\right]$, for some $n\in\N$, thus $\left[0,\frac{1}{n}f\right]$ is order compact (we use here that order compactness is inherited by order intervals). As $\left[0,\frac{1}{n}f\right]$ and $\left[0,f\right]$ are isomorphic lattices, the latter is also order compact.
\end{proof}

The implication (i)$\Rightarrow$(ii) was observed in~\cite{obr}.

\begin{example}
The Banach lattice $c$ is an atomic uniformly complete lattice which fails to be Dedekind complete. Hence, Dedekind completeness cannot be relaxed to uniform completeness in the theorem.
\qed\end{example}

Let $S$ be a Tychonoff space. It follows immediately from Lemma~\ref{cba} and Theorem~\ref{topuo} that the unit ball of $\Co_{b}\left(S\right)$ is order pre-compact iff $S$ has a dense set of isolated points. We also get the following characterization for order compactness.

\begin{corollary}
If $K$ is a compact Hausdorff space, then the unit ball of $\Co\left(K\right)$ is order compact iff $K$ is homeomorphic to the Stone-\v{C}ech compactification of a discrete space.
\end{corollary}
\begin{proof}
  Assume that $K=\beta\Omega$ for a set $\Omega$ viewed as a discrete topological space. Then $\Co\left(K\right)\simeq\ell_{\8}\left(\Omega\right)$, which is Dedekind complete and atomic. It follows from Theorem~\ref{topuo3} that the unit ball of $\Co\left(K\right)$ is order compact.

Conversely, if the unit ball of $\Co\left(K\right)$ is order compact, then one can deduce from Theorem~\ref{topuo3} that $\Co\left(K\right)$ is Dedekind complete and atomic. According to the comment above, $K$ has a dense set of isolated points. Moreover, it follows from Nakano-Stone theorem (see \cite[Corollary~1.51]{ab0}) that $K$ is extremally disconnected. It is straightforward that a compact Hausdorff extremally disconnected space with a dense set $U$ of isolated points is homeomorphic to~$\beta U$.
\end{proof}

Since the general strategy of proving the results of this section resembles one employed in the preceding section, it would be interesting to see if it is possible to directly deduce Theorems~\ref{topuo} and~\ref{topuo2} from Theorem~\ref{bain}. We succeed in doing that for Theorem~\ref{topuo3}. Recall that $f\in F_{+}$ is a \emph{component} of $e\in F_{+}$ if $f\wedge \left(e-f\right)=0$. Note that if $f$ is an atom, the only components of $f$ are $0$ and~$f$; moreover, this property characterizes atoms if $F$ is order complete. Note that components of a positive vectors form a Boolean algebra (this follows from \cite[Theorem~1.49]{ab}).

\begin{proof}[Alternative proof of Theorem~\ref{topuo3}]
In the current proof, Theorem~\ref{topuo2} was only used in proving the atomic clause in (ii)$\Rightarrow$(i). We will now deduce this from Corollary~\ref{bainc} instead. So suppose (ii) and let $e>0$. We need to show that there is an atom in $\left[0,e\right]$. Consider the Boolean algebra $P$ of all components of~$e$. We claim that $P$ is order compact.

Let $\left(f_{\alpha}\right)_{\alpha\in A}\subseteq P\subseteq\left[0,e\right]$. Since the latter is order compact, there is $g\in \left[0,e\right]$ such that $g_{\beta}\to g$, where $\left(g_{\beta}\right)_{\beta\in B}$ is a subnet of $\left(f_{\alpha}\right)_{\alpha\in A}$. In light of the completeness of $\left[0,e\right]$ this means that $g=\bigvee\limits_{\beta\in B}\bigwedge\limits_{\gamma\ge\beta}g_{\gamma}=\bigwedge\limits_{\beta\in B}\bigvee\limits_{\gamma\ge\beta}g_{\gamma}$. It is easy to see that a supremum or infimum (in $F$) of any collection of components of $e$ is again a component of~$e$, and so $P$ is complete, and $g_{\beta}\to g$ in~$P$. Since $\left(f_{\alpha}\right)_{\alpha\in A}$ was chosen arbitrarily, we conclude that $P$ is an order compact Boolean algebra.

Corollary~\ref{bainc} now guarantees that $P$ is atomic. Let $f\in P$ be an atom in~$P$. It is easy to see that $f$ only has two components in~$F$, and so is an atom in~$F$. Thus, $\left[0,e\right]$ contains an atom.
\end{proof}

\section*{Acknowledgements}

The first author was supported by project PID2021-122126NB-C32 funded by\linebreak MCIN/AEI/10.13039/501100011033 and by ``ERDF A way of making Europe'' and by Fundaci\'{o}n S\'{e}neca - Agencia de Ciencia y Tecnolog\'{\i}a de la Regi\'{o}n de Murcia (21955/PI/22).\medskip

The second author was supported by Pacific Institute for the Mathematical Sciences. The third author was supported by Natural Sciences and Engineering Research Council of Canada.\medskip

The authors would like to thank Banff International Research Station for hosting the ``Recent Advances in Banach lattices'' workshop, where the collaboration was initiated.


\begin{bibsection}
\begin{biblist}

\bib{acw}{article}{
   author={Abela, Kevin},
   author={Chetcuti, Emmanuel},
   author={Weber, Hans},
   title={On different modes of order convergence and some applications},
   journal={Positivity},
   volume={26},
   date={2022},
   number={1},
   pages={Paper No. 14, 22},
}

\bib{as}{article}{
   author={Abramovich, Yuri},
   author={Sirotkin, Gleb},
   title={On order convergence of nets},
   journal={Positivity},
   volume={9},
   date={2005},
   number={3},
   pages={287--292},
   issn={1385-1292},
}

\bib{ab0}{book}{
   author={Aliprantis, Charalambos D.},
   author={Burkinshaw, Owen},
   title={Locally solid Riesz spaces with applications to economics},
   series={Mathematical Surveys and Monographs},
   volume={105},
   edition={2},
   publisher={American Mathematical Society, Providence, RI},
   date={2003},
   pages={xii+344},
}

\bib{ab}{book}{
   author={Aliprantis, Charalambos D.},
   author={Burkinshaw, Owen},
   title={Positive operators},
   note={Reprint of the 1985 original},
   publisher={Springer, Dordrecht},
   date={2006},
   pages={xx+376},
}

\bib{aeg}{article}{
   author={Ayd\i n, Abdullah},
   author={Emelyanov, Eduard},
   author={Gorokhova, Svetlana},
   title={Full lattice convergence on Riesz spaces},
   journal={Indag. Math. (N.S.)},
   volume={32},
   date={2021},
   number={3},
   pages={658--690},
}

\bib{bb}{book}{
   author={Beattie, Ronald},
   author={Butzmann, Heinz-Peter},
   title={Convergence structures and applications to functional analysis},
   publisher={Kluwer Academic Publishers, Dordrecht},
   date={2002},
   pages={xiv+264},
}

\bib{bh}{article}{
   author={Bezhanishvili, Guram},
   author={Harding, John},
   title={On the proof that compact Hausdorff Boolean algebras are
   powersets},
   journal={Order},
   volume={33},
   date={2016},
   number={2},
   pages={263--268},
}

\bib{erz}{article}{
   author={Bilokopytov, Eugene},
   title={Characterizations of the projection bands and some order properties of the space of continuous functions},
   journal={\href{http://arxiv.org/abs/2211.11192}{arXiv:2211.11192}},
   date={2022},
}

\bib{erz0}{article}{
   author={Bilokopytov, Eugene},
   title={Locally solid convergences and order continuity of positive operators},
   journal={J. Math. Anal. Appl.},
   volume={528},
   date={2023},
   pages={Paper No. 127556},
}

\bib{et}{article}{
   author={Bilokopytov, Eugene},
   author={Troitsky, Vladimir G.},
   title={Order and uo-convergence in spaces of continuous functions},
   journal={Topology Appl.},
   volume={308},
   date={2022},
   pages={Paper No. 107999, 9},
}

\bib{ectv}{article}{
   author={Bilokopytov, Eugene},
   author={Conradie, Jurie},
   author={Troitsky, Vladimir G.},
   author={van der Walt, Jan Harm},
   title={Locally solid convergence structures},
   journal={in preparation},
}

\bib{dem}{article}{
   author={Dabboorasad, Yousef},
   author={Emelyanov, Eduard},
   author={Marabeh, Mohammad},
   title={Order convergence is not topological in infinite-dimensional
   vector lattices},
   journal={Uzbek Math. J.},
   date={2020},
   number={1},
   pages={159--166},
}

\bib{ellis}{book}{
   author={Ellis, John Thomas},
   title={Group topological convergence in completely distributive lattice-ordered groups},
   note={Thesis (Ph.D.)--Tulane University},
   publisher={ProQuest LLC, Ann Arbor, MI},
   date={1968},
   pages={56},
}

\bib{erne}{article}{
   author={Ern\'{e}, Marcel},
   title={Order-topological lattices},
   journal={Glasgow Math. J.},
   volume={21},
   date={1980},
   number={1},
   pages={57--68},
}

\bib{gtx}{article}{
   author={Gao, Niushan},
   author={Troitsky, Vladimir G.},
   author={Xanthos, Foivos},
   title={Uo-convergence and its applications to Ces\`aro means in Banach
   lattices},
   journal={Israel J. Math.},
   volume={220},
   date={2017},
   number={2},
   pages={649--689},
}

\bib{johnstone}{book}{
   author={Johnstone, Peter T.},
   title={Stone spaces},
   series={Cambridge Studies in Advanced Mathematics},
   volume={3},
   publisher={Cambridge University Press, Cambridge},
   date={1982},
   pages={xxi+370},
}

\bib{lipecky}{article}{
   author={Lipecki, Zbigniew},
   title={Compactness of order intervals in a locally solid linear lattice},
   journal={Colloq. Math.},
   volume={168},
   date={2022},
   number={2},
   pages={297--309},
}

\bib{obr}{book}{
   author={O'Brien, Michael},
   title={A Theory of Net Convergence with Applications to Vector Lattices},
   series={Ph.D. Thesis},
   publisher={University of Alberta},
   date={2021},
}

\bib{otv}{article}{
   author={O'Brien, Michael},
   author={Troitsky, Vladimir G.},
   author={van der Walt, Jan Harm},
   title={Net convergence structures with applications to vector lattices},
   journal={Quaest. Math.},
   volume={46},
   date={2023},
   number={2},
   pages={243--280},
}

\bib{pap1}{article}{
   author={Papangelou, Fredos},
   title={Order convergence and topological completion of commutative
   lattice-groups},
   journal={Math. Ann.},
   volume={155},
   date={1964},
   pages={81--107},
}

\bib{pap2}{article}{
   author={Papangelou, Fredos},
   title={Some considerations on convergence in abelian lattice-groups},
   journal={Pacific J. Math.},
   volume={15},
   date={1965},
   pages={1347--1364},
}

\bib{schaefer}{book}{
   author={Schaefer, Helmut H.},
   title={Banach lattices and positive operators},
   series={Die Grundlehren der mathematischen Wissenschaften},
   volume={Band 215},
   publisher={Springer-Verlag, New York-Heidelberg},
   date={1974},
   pages={xi+376},
}

\bib{sikorski}{book}{
   author={Sikorski, Roman},
   title={Boolean algebras.},
   series={},
   edition={3},
   publisher={Springer-Verlag New York, Inc., New York, },
   date={1969},
   pages={x+237},
}

\bib{taylor}{book}{
   author={Taylor, Mitchell},
   title={Unbounded convergences in vector lattices},
   series={M.Sc. Thesis},
   publisher={University of Alberta},
   date={2018},
}

\bib{vlad}{book}{
   author={Vladimirov, Denis A.},
   title={Boolean algebras in analysis},
   series={Mathematics and its Applications},
   volume={540},
   publisher={Kluwer Academic Publishers, Dordrecht},
   date={2002},
   pages={xxii+604},
}

\bib{weber}{article}{
   author={Weber, Hans},
   title={Group- and vector-valued $s$-bounded contents},
   conference={
      title={Measure theory, Oberwolfach 1983},
      address={Oberwolfach},
      date={1983},
   },
   book={
      series={Lecture Notes in Math.},
      volume={1089},
      publisher={Springer, Berlin},
   },
   date={1984},
   pages={181--198},
}

\bib{weinberg}{article}{
   author={Weinberg, Elliot Carl},
   title={Completely distributed lattice-ordered groups},
   journal={Pacific J. Math.},
   volume={12},
   date={1962},
   pages={1131--1137},
}

\bib{zsw}{article}{
   author={Zhang, Xiao-Dong},
   author={Schaefer, Helmut H.},
   author={Winkowska-Nowak, Kasia},
   title={Order continuity of locally compact Boolean algebras},
   journal={Positivity},
   volume={1},
   date={1997},
   number={4},
   pages={297--303},
}


\end{biblist}
\end{bibsection}
\end{document}